\documentclass{amsart}
\usepackage{amsmath,amsfonts,amsthm,amssymb,indentfirst,epic,url,graphics,needspace}

\newtheorem{theorem}{Theorem}
\newtheorem{lemma}[theorem]{Lemma}

\newtheorem{proposition}[theorem]{Proposition}
\newtheorem{definition}[theorem]{Definition}
\newtheorem{question}[theorem]{Question}

\renewcommand{\r}{\mathrm}

\begin{document}

\begin{center}
\texttt{Comments, corrections,
and related references welcomed, as always!}\\[.5em]
{\TeX}ed \today
\vspace{2em}
\end{center}

\title%
{A note on factorizations of finite groups}
\thanks{Readable at
\url{http://math.berkeley.edu/~gbergman/papers/}.
}

\subjclass[2010]{Primary: 20D60.
}
\keywords{factorization of a finite group as a product of subsets}

\author{George M. Bergman}
\address{University of California\\
Berkeley, CA 94720-3840, USA}
\email{gbergman@math.berkeley.edu}

\begin{abstract}
In Question 19.35 of the Kourovka Notebook \cite{Kourovka19},
M.\,H.\,Hooshmand asks whether, given a finite group $G$
and a factorization
$\r{card}(G)= n_1\ldots n_k,$ one can always
find subsets $A_1,\ldots,A_k$ of $G$ with $\r{card}(A_i)=n_i$
such that $G=A_1\ldots A_k;$ equivalently, such that the
group multiplication map $A_1\times\ldots\times\nolinebreak A_k\to G$ is
a bijection.

We show that for $G$ the alternating group on $4$ elements,
$k=3,$ and $(n_1,n_2,n_3) = (2,3,2),$ the answer is negative.
We then generalize some of the tools used in our proof,
and note an open question.
\end{abstract}
\maketitle

\section{The example.}\label{S.eg}

In this section we develop the example described in the Abstract.

\begin{definition}[after {\cite[\S1]{MHH_factor}}, cf.\ {\cite[p.\,6]{S+S}}]\label{D.decomp}
If $G$ is a group, $k$ a positive integer,
and $A_1,\ldots,A_k$ subsets of $G,$
we shall write $G = A_1\cdot\ldots\cdot A_k,$
and call this a $\!(k\!$-fold\textup{)} {\em factorization} of $G,$ if
the multiplication map $A_1\times\ldots\times A_k\to G$ is bijective.

\textup{(}Thus if $G$ is finite, this bijectivity condition can
alternatively be expressed, as in \cite[Question 19.35]{Kourovka19},
by the conditions $G=A_1\ldots A_k,$
and $\r{card}(G)=\r{card}(A_1)\ldots\r{card}(A_k).)$
\end{definition}

Note that if $G=A_1\cdot\ldots\cdot A_k$ is a $\!k\!$-fold
factorization of $G,$ then for $1\leq j< k,$
$G=(A_1\ldots A_j)\cdot\nolinebreak(A_{j+1}\ldots A_k)$
is a $\!2\!$-fold factorization.
So though the example we are working toward is a $\!3\!$-fold
factorization, the key to its proof will be the following
property of $\!2\!$-fold factorizations.

\begin{lemma}\label{L.div}
Let $G=A\cdot B$ be a factorization of a finite group.
Then the order of the subgroup of $G$ generated by $A$ is a
multiple of $\r{card}(A),$
and the order of the subgroup generated by $B$ is a
multiple of $\r{card}(B).$
\end{lemma}

\begin{proof}
Let $H$ be the subgroup of $G$ generated by $A.$
For each $b\in B,$ the set $A\,b$ is contained in
a single right coset of $H;$ hence every right coset of $H$
is a disjoint union of sets of cardinality $\r{card}(A),$
hence $\r{card}(H)$ is a multiple of $\r{card}(A).$
The statement about $B$ is seen in the same way.
\end{proof}

We will also use the following observation.

\begin{lemma}[{after \cite{MHH_factor}}]\label{L.e_in}
If $G=A_1\cdot \ldots\cdot A_k$ is a factorization of a
group $G,$ then for all $g,\,h\in G,$\linebreak
$(g\,A_1)\cdot A_2\cdot\ldots\cdot A_{k-1}\cdot (A_k\,h)$
is also a factorization of $G.$

Hence if for some positive integers
$n_1,\ldots,n_k,$ $G$ has a $\!k\!$-fold factorization with
$\r{card}(A_i)=n_i$ $(i=1,\ldots,k),$ it has such a
factorization in which $A_1$ and $A_k$ both contain the identity
element $e.$\qed
\end{lemma}

We can now prove

\begin{proposition}\label{P.2,3,2}
Let $G$ be the alternating group on $4$ elements, a
group of order~$12.$
Then $G$ has no factorization $A_1\cdot A_2\cdot A_3$
with $(\r{card}(A_1),\,\r{card}(A_2),\,\r{card}(A_3)) = (2,3,2).$
\end{proposition}

\begin{proof}
Recall that the elements of exponent $2$ in $G$ form a normal
subgroup $N\cong Z_2\times Z_2,$ that $G$ is an extension
of $N$ by a group of order $3$ which cyclically permutes the
three proper nontrivial subgroups of $N,$ and that all elements
of $G$ not in $N$ have order~$3.$

Suppose $G=A_1\cdot A_2\cdot A_3$ were a factorization
with $\r{card}(A_1)=2,$ $\r{card}(A_2)=3,$ $\r{card}(A_3)=2.$
By Lemma~\ref{L.e_in} we can assume without loss of generality
that $A_1$ and $A_3$ have the forms
$\{e,\,g\}$ and $\{e,\,h\}$ respectively.
The orders of the groups these sets generate are the
orders of $g$ and $h,$ and by Lemma~\ref{L.div}, are even.
But the only elements of $G$ of even order have order~$2,$ hence $A_1$
and $A_3$ are in fact subgroups (which may or may not be distinct).

Since $A_1$ and $A_3$ are contained in~$N,$
for $G=A_1\,A_2\,A_3$ to hold, $A_2$ must contain
representatives of all three cosets of $N$ in $G.$
Moreover, elements of $G$ act transitively on the set of $\!2\!$-element
subgroups of $N;$ so $A_2$ must contain
an element $g$ that conjugates $A_1$ to $A_3.$

Hence when we multiply out $A_1\,A_2\,A_3,$ the result contains
$A_1\,g\,A_3 = %
g\,A_3\,A_3.$
But the multiplication map $A_3\times A_3\to A_3$ is not
one-to-one; from which we see that the multiplication map
$A_1\times A_2\times A_3\to G$ cannot be one-to-one,
contradicting the definition of a factorization.
\end{proof}

This completes our negative answer to \cite[Question 19.35]{Kourovka19}
for $k=3.$
Note that for any $k,$ a negative example with
cardinalities $n_1,\dots,n_k$ yields negative
examples for all $k'>k,$ by keeping the same $G$ and $n_1,\ldots,n_k,$
and taking $n_{k+1}=\ldots=n_{k'}=1.$
So the one remaining open case is $k=2.$
Hooshmand posed the question in that case in {\cite{MHH_overflow}}, and
refers to it as a case of particular interest in \cite{Kourovka19}.
His paper \cite{MHH_factor} includes work on that case.

\section{Strengthening our lemmas}\label{S.stronger}

In the context of Lemma~\ref{L.div}, the order of the subgroup
$H$ of $G$ generated by $A$ can change on left-multiplying
$A$ by an element $g\in G,$ a fact we implicitly used
when we applied Lemma~\ref{L.e_in}
in the proof of Proposition~\ref{P.2,3,2}.
In the next result, modified versions of that subgroup are
noted whose orders are not so affected.
Also, while Lemma~\ref{L.div} is applicable only to the first
and last sets $A_1$ and $A_k$ in a factorization
$G = A_1\cdot \ldots\cdot A_k,$
part~(iii) below obtains a similar, though
weaker, condition on the cardinalities of the other $A_i.$
(This will be slightly improved in Lemma~\ref{L.div3}.)

\begin{lemma}\label{L.div2}
Let $A_1\cdot \ldots\cdot A_k$ be a factorization of a finite group $G.$
Then

\textup{(i)}  $\r{card}(A_1)$ divides
the order of the subgroup of $G$ generated
by the set $A_1^{-1} A_1 = \{g^{-1} h\mid g,h\in A_1\},$ which
can also be described as generated by any one of
the subsets $g^{-1} A_1$ $(g\in A_1).$
Moreover, that order is also the order of the subgroup generated
by $A_1\,A_1^{-1} = \{g\,h^{-1}\mid g,h\in A_1\},$ equivalently,
by any of the subsets $A_1\,g^{-1}$ $(g\in A_1).$

\textup{(ii)} Similarly, $\r{card}(A_k)$ divides
the order of the subgroup of $G$ generated
by $A_k\,A_k^{-1},$ equivalently,
by any of the subsets $A_k\,g^{-1}$ $(g\in A_k),$
and that order is also the order of the subgroup generated
by $A_k^{-1} A_k,$ equivalently,
by any of the subsets $g^{-1} A_k$ $(g\in A_k).$

\textup{(iii)} For $1<i<k,$ $\r{card}(A_i)$ divides
the order of the {\em normal subgroup} of $G$ generated
by $A_i^{-1} A_i,$ equivalently,
by any of the subsets $g^{-1} A_i$ $(g\in A_i),$ equivalently,
by $A_i\,A_i^{-1},$ or by any of the sets $A_i\,g^{-1}$ $(g\in A_i).$
\end{lemma}

\begin{proof}
(i)~~Combining Lemma~\ref{L.div} and Lemma~\ref{L.e_in}, we see that for
every $g\in A_1,$ $\r{card}(A_1)$
divides the order of the group generated by $g^{-1} A_1$
(an argument that we implicitly used in 
the proof of Proposition~\ref{P.2,3,2}).
Moreover, given $g,\,g'\in A_1,$ the
group generated by $g^{-1} A_1$ will contain
$(g^{-1} g')^{-1}(g^{-1} A_1) = {g'}^{-1} A_1;$
so the groups generated by $g^{-1} A_1$ are the same for
all $g\in A_1.$
Clearly their common value can also be described as the
group generated by $A_1^{-1} A_1,$ so the groups named
in the first sentence of~(i) are indeed equal.

The groups in the second sentence of~(i) are equal to one another
by the same argument.
Moreover, for any $g\in A_1,$
$A_1\,g^{-1} = g\,(g^{-1} A_1)\,g^{-1},$ so the group
generated by $A_1\,g^{-1}$ is conjugate in $G$ to the group
generated by $g^{-1} A_1.$
Hence the order of the group in the second sentence is the same
as that of the group in the first sentence.

(ii) holds by the same reasoning.

(iii) For each $i$ we similarly see that the {\em not} necessarily
normal subgroups generated by the sets named in the first half of~(iii)
are all equal, and are conjugate to the common value of those
generated by the sets named in the second half.
Hence the {\em normal} subgroups generated by these sets are all equal.
Let us call their common value $N.$

The condition $G = A_1\cdot \ldots\cdot A_k$ implies that
$G$ is the disjoint union of the sets $h\,A_i\,h'$ for
$h\in A_1\ldots A_{i-1},\ h'\in A_{i+1}\ldots A_k,$
and clearly each of these sets is wholly contained in one coset
of $N,$ namely $h\,N\,h' = hh'\,N = N\,hh'.$
Hence $N$ (and, indeed, every coset of $N)$
is the disjoint union of a family of
such sets, so $N$ indeed has order a multiple of $\r{card}(A_i).$
\end{proof}

We also note an easy strengthening of Lemma~\ref{L.e_in}.

\begin{lemma}\label{L.e_in2}
If $A_1\cdot\ldots\cdot A_k$ is a factorization of a
group $G,$ then for all $g_0,\,g_1,\ldots,g_k\in G,$\linebreak
$(g_0^{-1} A_1\,g_1)\,\cdot\,(g_1^{-1} A_1\,g_2)\,\cdot\,
\ldots\,\cdot\,(g_{k-1}^{-1} A_k\,g_k)$ is also a factorization of $G.$

In particular, if for some positive integers
$n_1,\ldots,n_k,$ $G$ has a $\!k\!$-fold factorization with
$\r{card}(A_i)=n_i$ $(i=1,\ldots,k),$ then it has such a
factorization in which all $A_i$ contain~$e.$\qed
\end{lemma}

A choice of $k+1$ elements $g_0,\ldots,g_k$ as above actually
gives one more degree of freedom than is needed to make
all the $A_i$ contain $e.$
This might be used to replace some particular term
by a chosen conjugate of itself.

Returning to Lemma~\ref{L.div2}, one may ask whether
in statement~(iii) thereof one can replace ``normal subgroup''
by ``subgroup'', as in~(i) and~(ii).
Probably not.
For though $G$ is the disjoint union of the
sets $h\,A_i\,h'$ referred to in the proof of~(iii), these lie in
cosets of different conjugates of $H;$ namely, $h\,A_i\,h'$ lies
in a right coset of $h\,H\,h^{-1}$ (and also in a left coset
of ${h'}^{-1} H\,h'),$ and such conjugates in general
partially overlap one another, so we can't get
a nice decomposition of any one of these cosets from our hypotheses.

The result~(iii) is very weak;
e.g., if $G$ is a simple group, it tells us nothing
that isn't evident from the definition of $A_1\cdot \ldots\cdot A_k$
being a factorization of $G.$
We give below a somewhat stronger, if not as easy to state, result.
To keep the statement from being too complicated, we shall
not use the ``strengthening'' gotten by replacing the sets in
our factorization by translates containing $e,$ but simply
understand that if this is desired, it can be achieved by
combining the result as stated with Lemma~\ref{L.e_in2}.

\begin{lemma}\label{L.div3}
In the context of Lemma~\ref{L.div2}\textup{(iii)}, let
$K,$ $H$ and $K'$ be, respectively, the subgroups of $G$ generated by
$A_1\ldots A_{i-1},$ by $A_i,$ and by $A_{i+1}\ldots A_k.$
Then $\r{card}(A_i)$ divides the order of the subgroup
$M$ of $G$ generated by the conjugates of $H$ by all members of $K,$
and also the order of the subgroup
$M'$ generated by the conjugates of $H$ by all members of $K'.$
\end{lemma}

\begin{proof}
As before, $G$ is the disjoint union of the sets $h\,A_i\,h'$ for
$h\in A_1\ldots A_{i-1},\ h'\in A_{i+1}\ldots A_k.$
Now $h\,A_i\,h'\subseteq h\,M\,h',$ which can be
rewritten as $M\,h\,h'$ because $M$ is
normalized by $h\in K.$
So each set $h\,A_i\,h'$ lies wholly in
one right coset of $M;$ so each right coset of $M$ is a
disjoint union of sets of cardinality $\r{card}(A_i),$
yielding the first of the asserted divisibility statements.
The second holds by the analogous reasoning.
\end{proof}

The two conditions on $\r{card}(A_i)$ obtained
in the above lemma differ, in general.
For instance, if $G$ is simple, and we take $k=3,$
let $A_1=\{e\},$ let $A_2$ be a proper nontrivial subgroup
$H$ of $G,$ and let $A_3$ be a set of left coset representatives
of $H$ in $G,$ then the multiplicative bound on $\r{card}(A_2)$
given by the first assertion is its actual cardinality,
while that given by the second is the order of $G.$

Though I have noted why we cannot expect that in this situation,
$\r{card}(A_i)$ will in fact divide $\r{card}(H),$
I don't know a counterexample, so let us record the question.
It clearly comes down to

\begin{question}\label{Q.A_2}
If a finite group $G$ has a factorization
$G= A_1\cdot A_2\cdot A_3,$ must $\r{card}(A_2)$ divide
the order of the subgroup $H$ of $G$ generated by $A_2$?
\end{question}


\end{document}